\documentclass[a4paper,12pt]{amsart}

\usepackage{amsfonts,amsmath} 
 
\usepackage{amssymb,latexsym}  
\usepackage[dvipsnames]{color}

\newtheorem{theorem}{Theorem}
\newtheorem{lemma}[theorem]{Lemma}  
  
\newtheorem{corollary}[theorem]{Corollary}

\theoremstyle{remark}  

\newcommand{\beq}{\begin{equation}}  
\newcommand{\enq}{\end{equation}}

\newcommand{\real}{\mathbb{R}}  
\newcommand{\R}{\mathbb{R}}

\newcommand{\loc}{\mbox{\footnotesize{loc}}}  
\newcommand{\cal}{\mathcal}
   
\newcommand{\neps}{\nu,\varepsilon}
\newcommand{\vare}{\varepsilon}

\newcommand{\abug}{\mathfrak{A^{\vare}}}
\newcommand{\bbug}{\mathfrak{B^{\vare}}}

\DeclareMathOperator{\curl}{curl}

\begin{document}  
\title[Vanishing viscosity limit around a small obstacle]
{Incompressible flow around a small obstacle and the vanishing viscosity limit}  
  
\author[]{Drago\c{s} Iftimie}  
\address{Universit{\'e} de Lyon,\\
Universit{\'e} Lyon 1\\
CNRS, UMR 5208 Institut Camille Jordan\\
B{\^a}timent du Doyen Jean Braconnier\\
43, blvd du 11 novembre 1918\\
F--69622 Villeurbanne Cedex\\
France}  
\email{dragos.iftimie@univ-lyon1.fr}  

\author[Milton Lopes {\em et al.}]{Milton C. Lopes Filho}  
\address{Depto. de Matem\'{a}tica, IMECC-UNICAMP \\ Cx. Postal 6065, Campinas SP 13083-970, Brazil} 
\email{mlopes@ime.unicamp.br}  
  
\author[]{Helena J. Nussenzveig Lopes}  
\address{Depto. de Matem\'{a}tica, IMECC-UNICAMP \\ Cx. Postal 6065, Campinas SP 13083-970, Brazil}
\email{hlopes@ime.unicamp.br}

\begin{abstract} In this article we consider viscous flow in the exterior of an obstacle satisfying the
standard no-slip boundary condition at the surface of the obstacle. We seek conditions under which
solutions of the Navier-Stokes system in the exterior domain converge to solutions of the Euler 
system in the full space when both viscosity and the size of the obstacle vanish. We prove that this 
convergence is true assuming two hypothesis: first, that the initial exterior domain velocity converges 
strongly in $L^2$ to the full-space initial velocity and second, that the diameter of the obstacle
is smaller than a suitable constant times viscosity, or, in other words, that the obstacle is sufficiently
small. The convergence holds as long as the solution to the limit problem is known to exist and stays 
sufficiently smooth. This work complements the study of incompressible 
flow around small obstacles, which has been carried out in \cite{iftimie,iln03,iln06}   
\end{abstract}

\maketitle

\section{Introduction}

The purpose of the present work is to study the asymptotic behavior of families of solutions of the
incompressible Navier-Stokes equations, in two and three space dimensions, in the exterior of a single
smooth obstacle, when both viscosity and the size of the obstacle become small. More precisely, let
$\Omega$ be a smooth and bounded domain in $\real^n$, $n=2,3$, such that $\Omega$ is connected and simply connected if $n=2$ and $\R^3\setminus\Omega$ is connected and simply connected if $n=3$. Let 
 $\vare > 0$ and set $\Pi_{\vare} = \real^n \setminus \vare \overline{\Omega}$. Let $u_0$ be a smooth, 
divergence-free vector field
in $\real^n$, which gives rise to a smooth solution $u$ of the Euler equations, defined on an interval $[0,T]$.
Let  $u^{\nu,\vare} \in L^{\infty}((0,T);L^2(\Pi_{\vare}))\cap C^0_w([0,T); L^2(\Pi_{\vare})\cap 
L^2((0,T);H^1_0(\Pi_{\vare}))$ be a weak Leray solution of the incompressible Navier-Stokes equations, with viscosity 
$\nu$, in $\Pi_{\vare}$, 
satisfying the no-slip boundary condition at $\partial \Pi_{\vare}$. 
We prove that there exists a constant $C= C(u_0,\Omega,T) >0$ such that if the following hypothesis holds:

[H] We have that 
$$\sup_{\vare \leq C\nu} \|u^{\nu,\vare}(\cdot,0) - u_0\|_{L^2(\Pi_{\vare})} \to 0,$$ 
as $\nu \to 0$, then we have that $\sup_{\vare \leq C\nu} \|u^{\nu,\vare} - u\|_{L^{\infty}((0,T);L^2(\Pi_{\vare})} \to 0,$
as $\nu \to 0$. Furthermore, if we assume that [H] occurs at a certain rate in $\nu$ we can obtain an
explicit error estimate in $L^2$.

In addition, we prove that if  we fix an initial vorticity 
$\omega_0$ in $\real^2$, smooth and compactly supported in $\real^2 \setminus \{0\}$ and consider 
$u^{\nu,\vare}(\cdot,0) = K^{\vare}[\omega_0] + m H^{\vare}$, where $K^{\vare}$ denotes the 
Biot-Savart operator in $\Pi_{\vare}$, while $H^{\vare}$ is the normalized generator of the
harmonic vector fields in $\Pi_{\vare}$, and $m = \int \omega_0$, then hypothesis [H] is satisfied. In the case of the dimension three, if we fix an initial vorticity $\omega_0$ in $\real^3$, smooth, divergence-free and compactly supported in $\real^3 \setminus \{0\}$ and consider 
$u^{\nu,\vare}(\cdot,0) = K^{\vare}[\omega_0]$, where $K^{\vare}$ again denotes the Biot-Savart
operator in $\Pi_{\vare}$, then it is proved in \cite{iftimie} that hypothesis [H] is satisfied. 
In both cases we have rates for the convergence of the initial data in such a way that  
$\|u^{\nu,\vare} - u\|_{L^2(\Pi_{\vare})} = \mathcal{O}(\sqrt{\nu})$ when $\nu \to 0$, uniformly in time.

A central theme in incompressible hydrodynamics is the vanishing viscosity limit, something naturally associated 
with the physical phenomena of turbulence and of boundary layers. In particular, a natural question to ask is 
whether the limiting flow associated with the limit of vanishing viscosity satisfies the incompressible Euler 
equations. This is known to be true in the absence of material boundaries, see \cite{Majda93,Chemin01} for the
two dimensional case and \cite{Kato72,Swann71} for the three dimensional case. Also, if the boundary conditions are of Navier type, see \cite{cmr98,lnp05,ip07,xz07}, noncharacteristic, see \cite{tw02} or for certain symmetric 2D flows,
see \cite{matsui95,lmn07,lmnt08}, convergence to an Euler solution remains valid. The most relevant case from the physical point of view corresponds to no slip boundary conditions. In this case, we have results on criteria for convergence to solutions of the Euler system, see \cite{Kato83,tw,Wang,Kelliher}, but the general problem remains wide open. To be a bit more precise, let us assume that $u^0$ is a solution of the Euler equations in the exterior domain 
$\Pi$ and that $u^{\nu}$ is a solution of the Navier-Stokes equations with viscosity $\nu$, with no-slip boundary condition in $\Pi$. Suppose further that $u^{\nu}$ and $u^0$ have the same initial velocity $v_0$ and that both $u^0$ and the family $\{u^{\nu}\}$ are smooth, defined on a fixed time interval $[0,T]$. It is easy to see that 
\[\delta E(\nu,t) \equiv \|u^0(\cdot,t) - u^{\nu}(\cdot,t)\|_{L^2(\Pi)} \]
is uniformly bounded in $\nu$ and $t \in [0,T]$ (by $2\|v_0\|_{L^2}$ if $v_0$ has finite energy), but it is not known whether $\delta E \to 0$ when $\nu \to 0$. In fact, given the experimentally and numerically observed behavior of high Reynolds number flows in the presence of boundaries, it is reasonable to conjecture that $\delta E$ does not, in general, 
vanish as $\nu \to 0$. Of course, this leaves open the possibility that $u^{\nu}$ might approach another solution of the Euler equations, different from $u^0$.     

This article contains an answer to the following question: can we make $\delta E$ small, by making 
both the viscosity and the obstacle small?   This problem was one of the main motivations underlying the authors' research on incompressible flow around small obstacles. Our previous results include the small obstacle limit for the 2D inviscid equations, see  \cite{iln03,Lopes07} and for the viscous equations, see \cite{iln06,iftimie}. The work we present here is a natural outgrowth of this research effort.

The remainder of this article is divided into three sections. In Section 2, we state and prove our main result, 
namely the convergence in the small viscosity and small obstacle limit, assuming convergence of the initial data.  
In Section 3 we study the problem of convergence of the initial data: for two space dimensions, we adapt  
techniques developed in our previous work, while for three space dimensions we report to work by D. Iftimie and J. Kelliher. In Section 4, we interpret the smallness condition on the obstacle as the condition that the {\it local} Reynolds number stays below a certain (small) constant. In addition, still in Section 4, we obtain an enstrophy estimate for the wake generated by the small obstacle and we list some open problems.  
       
\section{Main theorem}

We use the notation from the introduction to state and prove our main result. We consider the 
initial-value problem for the Navier-Stokes equations in $\Pi_{\vare}$, with no-slip boundary condition,
given by:
\begin{equation} \label{neps}
\left\{ \begin{array}{ll}
\partial_t u^{\neps} + u^{\neps} \cdot \nabla u^{\neps} = -\nabla p^{\neps} + \nu \Delta u^{\neps}, & \mbox{ in }
\Pi_{\vare} \times (0,\infty)\\
\mbox{div }u^{\neps} = 0 & \mbox{ in } \Pi_{\vare} \times [0,\infty)\\
u^{\neps}(x,t) = 0 & \mbox{ for } x \in \partial \Pi_{\vare}, \;\;\; t>0\\
u^{\neps}(t=0) = \vartheta^{\vare}(x) & \mbox{ for } x\in \Pi_{\vare},\;\;\;t=0.
\end{array} \right. 
\end{equation}

We assume that the initial velocity $\vartheta^{\vare} \in L^2_{\loc}(\Pi_{\vare})$ is divergence-free and tangent to 
$\partial \Pi_{\vare}$, but we do not assume that it satisfies the no-slip boundary condition. In three dimensions we assume further that $\vartheta^{\vare} \in L^2(\Pi_{\vare})$. Under these hypothesis' it was shown by H. Kozono and M. Yamazaki, see \cite{kozono-yamazaki}, that, in two dimensions, there is a unique global strong solution to \eqref{neps} with initial velocity $\vartheta^{\vare}$, while, in three dimensions, there is a global Leray weak solution $u^{\neps}$ of \eqref{neps}, see \cite{hopf}. 
More precisely, in three dimensions there exists 
\[u^{\neps} \in L^{\infty}([0,\infty);L^2(\Pi_{\vare})) \cap C^0_w([0,\infty);L^2(\Pi_{\vare}))\cap 
L^2_{\loc}([0,\infty);H^1_0(\Pi_{\vare}))\]
such that $u^{\neps}$ is a distributional solution of \eqref{neps} and the following energy inequality holds true:
\begin{equation}   \label{energ}
\|u^{\neps}(t)\|_{L^2(\Pi_{\vare})}+2\nu\int_0^t \|\nabla u^{\neps}(s)\|_{L^2(\Pi_{\vare})}\,ds 
\leq \|u^\vare_0\|_{L^2(\Pi_{\vare})}\qquad\forall t\geq0.  
 \end{equation}

Both $\vartheta^{\vare}$ and $u^{\neps}(\cdot,t)$ are defined only in $\Pi_{\vare}$, but we will consider
them as defined on the whole space by extending them to be identically zero inside $\vare \Omega$. 

Let $u_0$ be a smooth, divergence-free vector field defined in all $\real^n$, and let $u$ be the corresponding smooth solution of the Euler equations; in two dimensions $u$ is globally defined while in three dimensions it is defined, at least, on an interval $[0,T]$.

We are now ready to state our main result. 

\begin{theorem} \label{mainthm} Assume that 
\[\|\vartheta^{\vare} - u_0 \|_{L^2(\real^n)} \to 0 \mbox{ as } \vare \to 0.\]
Fix $T>0$, arbitrary if $n=2$, and smaller than the time of existence of the smooth Euler solution if $n=3$.  
Then there exists a constant $C_1 = C_1(\Omega,u_0,T)$ such that, if $\vare \leq C_1 \nu$, then 
\[\|u^{\neps}(\cdot,t) - u(\cdot,t)\|_{L^2(\real^n)} \to 0 \mbox{ as } \nu \to 0.\]

Moreover, if we assume that $\|\vartheta^{\vare} - u_0 \|_{L^2(\real^n)} = \mathcal{O}(\sqrt{\nu})$,
 then there exists also $C_2=C_2(T,u_0,\Omega)$ such that 
\[\|u^{\neps}(\cdot,t) - u(\cdot,t)\|_{L^2(\real^n)} \leq C_2 \sqrt{\nu},\]
for all $0<\vare<C_1 \nu$ and all $0\leq t \leq T$. 
\end{theorem}

Before we proceed with the proof, we require two technical lemmas.
To state the first lemma we must introduce some notation. As in the statement of the theorem, $u$ denotes
the smooth Euler solution in $\real^n$. In
dimension two, we denote by $\psi = \psi(x,t)$ the 
stream function for the velocity field $u$, chosen so that
$\psi(0,t) = 0$. In dimension three, $\psi$ denotes the unique
divergence free vector field vanishing for $x=0$ and  whose curl is
$u$. In other words, we set
\begin{equation*}
 \psi(x,t)=\int_{\R^2}\frac{(x-y)^\perp\cdot u(y,t)}{2\pi|x-y|^2}dy
+\int_{\R^2}\frac{y^\perp\cdot u(y,t)}{2\pi|y|^2}dy 
\end{equation*}
in dimension two  so that $u=\nabla^\perp\psi$ and 
\begin{equation*}
 \psi(x,t)=-\int_{\R^3}\frac{x-y}{4\pi|x-y|^3}\times u(y,t)dy
-\int_{\R^3}\frac{y}{4\pi|y|^3}\times u(y,t)dy
\end{equation*}
in dimension three so that $u=\curl\psi$. In both two and three dimensions one has that $\psi$ and
$\nabla\psi$ are uniformly bounded on the time interval $[0,T]$.


Let $R>0$ be such that the ball of radius $R$, centered at the origin, contains $\Omega$.
Let $\varphi = \varphi(r)$ be a smooth function on $\real_+$ such that
$\varphi(r) \equiv 0$ if $0 \leq r \leq R+1$, $\varphi \geq 0$ and $\varphi(r) \equiv 1$ 
if $r \geq R+2$. Set $\varphi^{\vare} = \varphi^{\vare}(x) = \varphi(|x|/\vare)$ and 
\[u^{\vare} = \nabla^{\perp}(\varphi^{\vare}\psi)\]
in dimension two and
\[u^{\vare} = \curl(\varphi^{\vare}\psi)\]
in dimension three. In both dimensions two and three, the vector field $u^{\vare}$ 
is divergence free and vanishes in a neighborhood of the boundary.

We also re-define the pressure $p=p(x,t)$ from the Euler equation in $\R^n$ with data $u_0$ so that $p(0,t) = 0$.  

\begin{lemma} \label{epsests} Fix $T>0$. There exist constants $K_i>0$, $i=1,\ldots,5$ such that,
for any $0 < \vare < \vare_0$ and any $0 \leq t < T$ we have:
\begin{enumerate}
\item $\|\nabla u^{\vare}\|_{L^2}^2 \leq K_1$,
\item $\|u^{\vare}\|_{L^{\infty}} \leq K_2$,
\item $\|u^{\vare}-u\|_{L^2} +\|u^{\vare}-\varphi^{\vare}u\|_{L^2} \leq K_3 \vare$,
\item $\|\nabla\psi \nabla\varphi^{\vare}\|_{L^{\infty}} + \|\psi \nabla^2 \varphi^{\vare} \|_{L^{\infty}} \leq K_4/\vare$,
\item $\|p \nabla \varphi^{\vare}\|_{L^2} + \|\partial_t \psi \nabla \varphi^{\vare}\|_{L^2} \leq K_5 \vare$.
\end{enumerate}
\end{lemma}

Above, we used the notation $\|\nabla\psi
\nabla\varphi^{\vare}\|_{L^{\infty}}=\sum_{i,j}\|\partial_i\psi
\partial_j\varphi^{\vare}\|_{L^{\infty}} $ in dimension two and 
$\|\nabla\psi
\nabla\varphi^{\vare}\|_{L^{\infty}}=\sum_{i,j,k}\|\partial_i\psi_k
\partial_j\varphi^{\vare}\|_{L^{\infty}} $ in dimension three. Similar
notations were used for
the other terms. 

\begin{proof}
Some of the inequalities above can be improved in dimension
three. However, it turns out that these improvements have no
effect on the statement of Theorem \ref{mainthm}. Therefore,  to
avoid giving separate proofs in dimension three we chose to state
these weaker estimates.

Recall that both $u$ and $\nabla u$ are uniformly bounded. 
First we write
\[\partial_i u^{\vare} = \partial_i \nabla^{\perp}(\varphi^{\vare}
\psi) = u \partial_i \varphi^{\vare}
+\partial_i\psi\nabla^\perp\varphi^\vare 
+ \psi \partial_i\nabla^\perp  \varphi^{\vare} + \varphi^{\vare}
\partial_iu\]
in dimension two and 
\[\partial_i u^{\vare} = \partial_i \curl(\varphi^{\vare}
\psi) = u \partial_i \varphi^{\vare} 
+\nabla\varphi^\vare\times\partial_i\psi
+ \partial_i\nabla \varphi^{\vare}\times\psi + \varphi^{\vare}
\partial_iu\]
in dimension three.
The supports of the first three terms of the right-hand sides of the
relations above are contained in the annulus $\vare(R+1) < |x| < \vare(R+2)$,
whose Lebesgue measure is $\mathcal{O}(\vare^2)$. Furthermore, 
$|\nabla\varphi^{\vare}| = \mathcal{O}(1/\vare)$, $|\nabla^2 \varphi^{\vare}| = \mathcal{O}(1/\vare^2)$
and $|\psi(x,t)| = \mathcal{O}(\vare)$ for $|x| < \vare(R+2)$, since $\psi(0,t) = 0$. 
Taking $L^2$ norms in the expressions above gives the first estimate.

Next we observe that $u^{\vare} = \varphi^{\vare} u + \psi
\nabla^{\perp} \varphi^{\vare}$ or $u^{\vare} = \varphi^{\vare} u +
\nabla \varphi^{\vare}\times\psi$. Clearly $\varphi^{\vare} u$ 
is bounded and to bound the second term, we use again that $\psi(0,t) = 0$,
which proves the second estimate. For the third estimate, observe that
$u^{\vare} - u$ and $u^{\vare}-\varphi^{\vare}u$ are bounded,
as we have just proved, and have support in the ball $|x| < \vare(R+2)$. For the fourth estimate,
we use again that $\psi(0,t) = 0$. The last estimate follows from two facts: that the functions
whose $L^2$-norm we are estimating have support on the ball $|x| < \vare(R+2)$ and that
they are both bounded, since $p(0,t) = 0$ and $\psi_t(0,t) = 0$.  
\end{proof}

We also require a modified Poincar\'{e} inequality, stated below. This is fairly standard, but we include
a sketch of the proof for completeness. 

\begin{lemma} \label{poincare}
Let $\Omega$ be the obstacle under consideration and let $R$ be such that $\Omega \subset B_R$. Consider the scaled
obstacles $\vare \Omega$ and the exterior domains $\Pi_{\vare}$. Then, if 
$W \in H^1_0(\Pi_{\vare})$ we have
\[\|W\|_{L^2(\Pi_{\vare} \cap B_{(R+2)\vare} )} \leq K_6 \,\vare \, \|\nabla W\|_{L^2(\Pi_{\vare} \cap B_{(R+2)\vare})}.\]
\end{lemma}
  
\begin{proof}
The proof proceeds in two steps. First we establish the result in the case $\vare = 1$.
Suppose, by contradiction, that there exists a sequence $\{W^k\} \subset H^1_0(\Pi_1)$ 
such that $\|W^k\|_{L^2(\Pi_1 \cap B_{R+2})} > k \|\nabla W^k\|_{L^2(\Pi_1 \cap B_{R+2})}$. Set
\[V^k \equiv \frac{W^k}{\|W^k\|_{L^2(\Pi_1 \cap B_{R+2})}}.\]
Then $V^k \in H^1_0(\Pi_1)$, with unit $L^2$ norm
in $\Pi_1 \cap B_{R+2}$, while the $L^2$-norm of its gradient on $\Pi_1 \cap B_{R+2}$ vanishes as $k\to\infty$. Thus, passing to a subsequence
if necessary, $V^k \to V$, weakly in $H^1$ and strongly in $L^2$ (on $\Pi_1 \cap B_{R+2}$) so that 
$\|V\|_{L^2(\Pi_1 \cap B_{R+2})} = 1$ and $\nabla V = 0$. Since this set is connected, it follows that
$V$ is constant in $\Pi_1 \cap B_{R+2}$. By continuity of the trace, the trace of $V$ on $\Gamma = \partial \Pi_1$ must vanish, which shows that $V \equiv 0$ in $\Pi_1 \cap B_{R+2}$, a contradiction.   

We conclude the proof with a scaling argument. Let $W \in H^1_0(\Pi_{\vare})$ and set $Y = Y(x) = W(\vare x)$. 
Then $Y \in H^1_0(\Pi_1)$. Using the first step we deduce that there exists a constant $K_6$ such that 
$\|Y\|_{L^2(\Pi_1 \cap B_{R+2} )} \leq K_6 \|\nabla Y\|_{L^2(\Pi_1 \cap B_{R+2})}$.  Undoing the scaling we find:
\[\|Y\|_{L^2(\Pi_1 \cap B_{R+2} )}^2 =   \int_{\Pi_1 \cap B_{R+2}} |W(\vare x)|^2 \, dx   = 
\frac{\|W\|_{L^2(\Pi_{\vare} \cap B_{(R+2)\vare})}^2}{\vare^{n}};\]
\[\|\nabla Y\|_{L^2(\Pi_1 \cap B_{R+2})}^2 =  \int_{\Pi_1 \cap B_{R+2}} \vare^2|\nabla W \, (\vare x)|^2 \, dx  
 = \vare^{2-n}\|\nabla W\|_{L^2(\Pi_{\vare} \cap B_{(R+2)\vare})}^2.\]
The desired result follows immediately. 

\end{proof}

We are now ready to prove Theorem \ref{mainthm}. 

{\bf Proof of Theorem \ref{mainthm}:}
We begin by noting that, since $u$ is a smooth solution of the Euler equations in $\real^n \times [0,T]$, it follows
that 
\begin{equation} \label{intest}
\|u^{\neps}(\cdot,t)-u(\cdot,t)\|_{L^2(\vare \Omega)} \equiv \|u(\cdot,t)\|_{L^2(\vare\Omega)} \leq C \vare \leq C\sqrt{\nu},
\end{equation}
if $\vare < C \nu$. Hence it remains only to estimate the $L^2$-norm of the difference in $\Pi_{\vare}$, which we 
do below.

We first give the proof in two dimensions and then we 
indicate how the proof should be adapted to three dimensions.
\subsection{Case $n=2$}
The vector field $u^{\vare}$ is divergence free and satisfies the equation
\[u^{\vare}_t = -\varphi^{\vare}u \cdot \nabla u - \varphi^{\vare} \nabla p 
+ \partial_t \psi \nabla^{\perp} \varphi^{\vare}.\]
We set $W^{\neps} \equiv u^{\neps} - u^{\vare}$. The vector field
$W^{\neps}$ is divergence free, vanishes on the boundary and satisfies:
\[\partial_t W^{\neps} - \nu \Delta W^{\neps} = - u^{\neps} \cdot \nabla u^{\neps} - \nabla p^{\neps} + \nu \Delta u^{\vare} + 
\varphi^{\vare} u \cdot \nabla u + \varphi^{\vare} \nabla p - \partial_t \psi \nabla^{\perp} \varphi^{\vare}.\]

We perform an energy estimate, multiplying this equation by $W^{\neps}$ and integrating over $\Pi_{\vare}$.
We obtain
\begin{multline}\label{enid}
\frac{1}{2} \frac{d}{dt} \|W^{\neps}\|_{L^2}^2 + \nu \|\nabla W^{\neps}\|_{L^2}^2 =
-\nu \int_{\Pi_{\vare}} \nabla W^{\neps} \cdot \nabla u^{\vare}\, dx \\  
- \int_{\Pi_{\vare}} W^{\neps} \cdot [(u^{\neps} \cdot \nabla) u^{\neps}]\, dx +
\int_{\Pi_{\vare}} W^{\neps} \cdot [(\varphi^{\vare}u \cdot \nabla) u]\, dx\\
+ \int_{\Pi_{\vare}} W^{\neps} \cdot \varphi^{\vare} \nabla p\, dx 
- \int_{\Pi_{\vare}} W^{\neps} \cdot \partial_t \psi \nabla^{\perp} \varphi^{\vare}\, dx.
\end{multline}
We will examine each one of the five terms on the right-hand-side of identity \eqref{enid}.
We look at the first term. We use Cauchy-Schwarz and Young's inequalities followed by Lemma \ref{epsests}, item (1), to obtain
\begin{equation} \label{term1}
\Bigl|\nu \int_{\Pi_{\vare}} \nabla W^{\neps} \cdot \nabla u^{\vare} \, dx \Bigr| \leq 
\frac{\nu}{2}\Bigl(\|\nabla W^{\neps}\|_{L^2}^2 + K_1 \Bigr).
\end{equation}
 
Next we look at the second and third terms together. We write
\begin{align*}
|\cal{I}| \equiv
& \Bigl| - \int_{\Pi_{\vare}} W^{\neps} \cdot [(u^{\neps} \cdot \nabla) u^{\neps}]\, dx +
\int_{\Pi_{\vare}} W^{\neps} \cdot [(\varphi^{\vare}u \cdot \nabla)
u]\, dx \Bigr|\\
=& \Bigl| - \int_{\Pi_{\vare}} W^{\neps} \cdot [((W^{\neps}+u^{\vare}) \cdot \nabla)(W^{\neps}+u^{\vare})]\, dx +
\int_{\Pi_{\vare}} W^{\neps} \cdot [(\varphi^{\vare}u \cdot \nabla)
u]\, dx \Bigr|\\
=& \Bigl| - \int_{\Pi_{\vare}} W^{\neps} \cdot [(W^{\neps} \cdot \nabla)u^{\vare}]\, dx 
- \int_{\Pi_{\vare}} W^{\neps} \cdot [(u^{\vare} \cdot \nabla)u^{\vare}]\, dx \\
&\hskip 5cm + \int_{\Pi_{\vare}} W^{\neps} \cdot [(\varphi^{\vare}u \cdot \nabla)
u]\, dx \Bigr|,
\end{align*}
where we used the fact that $\int W^{\neps} \cdot [((W^{\neps}+u^{\vare}) \cdot \nabla) W^{\neps}] = 0$. Finally,  
we add and subtract $\int W^{\neps} \cdot [(u^{\vare} \cdot \nabla) u]$ to obtain
{\allowdisplaybreaks
\begin{align*}
|\cal{I}| =& \Bigl|\int_{\Pi_{\vare}} \bigl\{ -W^{\neps} \cdot (W^{\neps} \cdot \nabla u^{\vare}) + 
W^{\neps} \cdot [u^{\vare} \cdot \nabla (u-u^{\vare})] \\
&\hskip 5cm + W^{\neps} \cdot [((\varphi^{\vare}u-u^\vare) \cdot \nabla)
u]\bigr\}\, dx   \Bigr|\\
\leq& \Bigl|\int_{\Pi_{\vare}} W^{\neps} \cdot (W^{\neps} \cdot \nabla u^{\vare}) \, dx \Bigr| + \Bigl|\int_{\Pi_{\vare}} W^{\neps} \cdot [u^{\vare} \cdot \nabla (u-u^{\vare})]\, dx \Bigr| \\
&\hskip 5cm + \Bigl|\int_{\Pi_{\vare}} W^{\neps} \cdot [((\varphi^{\vare}u-u^\vare) \cdot \nabla)
u]\, dx   \Bigr|\\
=&\Bigl|\int_{\Pi_{\vare}} W^{\neps} \cdot (W^{\neps} \cdot \nabla u^{\vare})\, dx \Bigr| + 
\Bigl|\int_{\Pi_{\vare}} (u-u^{\vare}) \cdot [(u^{\vare} \cdot \nabla)W^{\neps}]\, dx \Bigr| \\
&\hskip 5cm + \Bigl|\int_{\Pi_{\vare}} W^{\neps} \cdot [((\varphi^{\vare}u-u^\vare) \cdot \nabla)
u]\, dx   \Bigr|\\
\leq&  \Bigl|\int_{\Pi_{\vare}} (W^{\neps} \cdot \nabla u^{\vare}) \cdot W^{\neps} \, dx 
\Bigr|+  \|u-u^{\vare}\|_{L^2} \|u^{\vare}\|_{L^{\infty}} \|\nabla W^{\neps}\|_{L^2} \\
&\hskip 5cm +\|W^{\neps}\|_{L^2} \|\varphi^{\vare}u-u^{\vare}\|_{L^2}\|\nabla u\|_{L^\infty}.
\end{align*}
}

For each $i=1,2$ we have that
\begin{equation}
  \label{dij}
\partial_{i} u^{\vare} = \partial_i\psi \nabla^{\perp} \varphi^{\vare} +
\psi \partial_{i}\nabla^{\perp} \varphi^{\vare} + \partial_{i} \varphi^{\vare} u +
\varphi^{\vare} \partial_{i} u.   
\end{equation}

Therefore,
\begin{multline}\label{term23a}
|\cal{I}| \leq  
(\|\nabla\psi  \nabla\varphi^{\vare}\|_{L^{\infty}} + 
\|\psi \nabla^2 \varphi^{\vare} \|_{L^{\infty}}) \|W^{\neps}\|_{L^2(A_{\vare})}^2+
 \|\varphi^{\vare} \nabla u \|_{L^{\infty}} \|W^{\neps}\|_{L^2}^2 
\\
+\|u-u^{\vare}\|_{L^2} \|u^{\vare}\|_{L^{\infty}} \|\nabla W^{\neps}\|_{L^2} + 
\|W^{\neps}\|_{L^2} \|\varphi^{\vare}u-u^{\vare}\|_{L^2}\|\nabla u\|_{L^\infty},
\end{multline}
where $A_{\vare}$ is the set $\Pi_{\vare} \cap B_{(R+2)\vare}$, which contains the support of
$\nabla\varphi^{\vare}$.

Hence, using Lemma \ref{epsests}, items (2), (3) and (4), together with Lemma \ref{poincare}, in the inequality \eqref{term23a}, we find
\begin{multline}\label{term23}
|\cal{I}| \leq  \frac{K_4}{\vare} \vare^2 K_6^2 \| \nabla W^{\neps}\|_{L^2}^2 
+ \|\varphi^{\vare} \nabla u \|_{L^{\infty}} \|W^{\neps}\|_{L^2}^2 \\
+ K_2 K_3 \vare \| \nabla W^{\neps}\|_{L^2}^2
+ K_3 \vare \|\nabla u\|_{L^\infty}\|W^{\neps}\|_{L^2}.
\end{multline}

Next we look at the fourth and fifth terms in \eqref{enid}. Recall that we chose the pressure $p$ in such a way that 
$p(0,t)=0$. We find
\[|\cal{J}| \equiv \Bigl| \int_{\Pi_{\vare}} W^{\neps} \cdot \varphi^{\vare} \nabla p\, dx 
- \int_{\Pi_{\vare}} W^{\neps} \cdot \partial_t \psi \nabla^{\perp} \varphi^{\vare}\, dx \Bigr|\]
\[\leq \Bigl| \int_{\Pi_{\vare}} W^{\neps} \cdot p \nabla \varphi^{\vare}\, dx \Bigr| 
+ \Bigl|\int_{\Pi_{\vare}} W^{\neps} \cdot \partial_t \psi \nabla^{\perp} \varphi^{\vare}\, dx \Bigr|.\]
We estimate each term above to obtain, using Lemma \ref{epsests} item (5),
\begin{equation} \label{term45}
|\cal{J}| \leq (\|p\nabla\varphi^{\vare}\|_{L^2} + \|\partial_t \psi \nabla^{\perp} \varphi^{\vare}\|_{L^2})
\|W^{\neps}\|_{L^2} \leq K_5 \vare \|W^{\neps}\|_{L^2}.
\end{equation}

We use estimates \eqref{term1}, \eqref{term23} and \eqref{term45} in the energy identity \eqref{enid} to deduce that
\begin{equation} \label{bliblu}
\frac{1}{2}\frac{d}{dt}\|W^{\neps}\|_{L^2}^2 + \nu \|\nabla W^{\neps}\|_{L^2}^2 \leq
\frac{\nu}{2} \|\nabla W^{\neps}\|_{L^2}^2 + \frac{\nu}{2}K_1 + K_2 K_3 \vare \|\nabla W^{\neps}\|_{L^2} 
\end{equation}
\[+ K_4 K_6^2 \vare \|\nabla W^{\neps}\|_{L^2}^2 +
\|\varphi^{\vare} \nabla u \|_{L^{\infty}} 
 \|W^{\neps}\|_{L^2}^2 + \vare (K_5 +K_3\|\nabla u\|_{L^\infty})\|W^{\neps}\|_{L^2}\]
\[\leq \frac{\nu}{2} \|\nabla W^{\neps}\|_{L^2}^2 + \frac{\nu}{2}K_1 +  \frac{\nu}{4} \|\nabla W^{\neps}\|_{L^2}^2 
+ K_2^2 K_3^2 \frac{\vare^2}{\nu} 
+ K_4 K_6^2 \vare \|\nabla W^{\neps}\|_{L^2}^2 \]
\[+ K_0 \|W^{\neps}\|_{L^2}^2  + \frac{\|W^{\neps}\|_{L^2}^2}{2} +
\frac{\widetilde{K}_5^2 \vare^2}{2}.\]
Above we have used the notation $K_0 = \sup_{\vare}  \|\varphi^{\vare}
\nabla u \|_{L^{\infty}}$ and $\widetilde{K}_5=K_5 +K_3\|\nabla u\|_{L^\infty} $.

At this point we choose $\vare$ so that 
\begin{equation} \label{smcond}
0< \vare < \min{\Bigl\{\vare_0 \, , \, \frac{\nu}{8 K_4 \, K_6^2} \Bigr\}}.
\end{equation}

With this choice, letting $y=y(t)=\|W^{\neps}\|_{L^2}^2$, we obtain
\begin{equation}
  \label{dty}
\frac{dy}{dt} \leq \nu K_1 + 2K_2^2 K_3^2 \frac{\vare^2}{\nu} + \widetilde{K}_5^2 \vare^2 + (2K_0 + 1) y\leq C'_1\nu + C'_2 y.  
\end{equation}

From Gronwall's inequality it follows that
\begin{equation} \label{gron}
\|u^{\neps} - u^{\vare}\|^2_{L^2(\Pi_{\vare})} \leq C(T,u_0,\Omega) \left( \nu + 
\|\vartheta^{\vare} - u_0^{\vare}\|^2_{L^2(\Pi_{\vare})} \right). 
\end{equation} 

If $\|\vartheta^{\vare} - u_0\|_{L^2} \to 0$ then  it follows from \eqref{gron} together with 
Lemma \ref{epsests}, item (3), and from \eqref{intest}, that $\|u^{\neps} - u\|_{L^2} \to 0$, 
as desired, where the constant $C_1$ can be chosen to be $(8 K_4 K_6^2)^{-1}$.
If we assume further that $\|\vartheta^{\vare} - u_0\|_{L^2}= \mathcal{O}(\sqrt{\nu})$
then the second part of the statement of Theorem \ref{mainthm} easily follows. This concludes the proof in the 
two dimensional case.

\subsection{Case $n=3$}

The proof in dimension three is similar to the previous one. There are
two differences: notation and the justification that we can multiply
the equation of $W^{\neps}$ by $W^{\neps}$. 

First, about notation. One has to replace everywhere the term
$\partial_t\psi\nabla^\perp\varphi^\vare$ by $
\nabla\varphi^\vare\times\partial_t \psi$ and also relation  \eqref{dij} becomes
\begin{equation*}
\partial_{i} u^{\vare} = \nabla\varphi^{\vare}\times\partial_i\psi +
\partial_{i}\nabla \varphi^{\vare}\times\psi + \partial_{i} \varphi^{\vare} u +
\varphi^{\vare} \partial_{i} u.   
\end{equation*}
These two modifications are just changes of notations. These new terms
are of the same type as the old ones, so the estimates that follow are
not affected.

Second, we multiplied the equation of
$W^{\neps}$ by $W^{\neps}$. The solution $u^{\neps}$, and
therefore $W^{\neps}$ too,  is not better
than $L^\infty(0,T;L^2(\Pi_\vare))\cap L^2(0,T;H^1(\Pi_\vare))$. But it is
well-known that some of the trilinear terms that appear  when
multiplying the equation of
$W^{\neps}$ by $W^{\neps}$ are not well defined in dimension three
with this regularity only. In other words, one cannot multiply
directly the equation of $W^{\neps}$ by $W^{\neps}$. Nevertheless,
there is a classical trick that allows us to perform this multiplication if
the weak solution $u^{\neps}$ verifies the energy inequality. What we
are trying to do, is to subtract the equation of $u^\vare$ from the
equation of $u^{\neps}$ and to multiply the result by
$u^{\neps}-u^\vare$. This is the same as multiplying the equation of
$u^{\neps}$ by $u^{\neps}$, adding the equation of $u^\vare$ times
$u^\vare$ and subtracting the equation of $u^{\neps}$ times
$u^\vare$ and the equation of $u^\vare$ times $u^{\neps}$. Since
$u^\vare$ is smooth, all these operations are legitimate except for
the multiplication of the equation of $u^{\neps}$ by
$u^{\neps}$. Formally, multiplying the equation of $u^{\neps}$ by
$u^{\neps}$ and integrating in space and time from 0 to $t$ yields the
energy equality, \textit{i.e.} relation \eqref{energ}  where the sign $\leq$ is
replaced by =. Since we assumed that the energy inequality holds true, the above
operations are justified provided that the relation we get at the end
is an inequality instead of an equality. But an inequality is, of
course, sufficient for our purpose. Finally, to be completely
rigorous, one has to integrate in time from the begining. That is, we
would obtain at the end relation \eqref{dty} integrated in time. Clearly, the
result of the application of the Gronwall lemma in \eqref{dty} is the
same as in \eqref{dty} integrated in time. This completes the proof in
dimension three. 

{\bf Remark:} The proof above is closely related to the proof of Kato's criterion
for the vanishing viscosity limit in bounded domains, see \cite{Kato83}. Both results
are based on estimating the difference between Navier-Stokes solutions and Euler solutions
by means of energy methods. In Kato's argument, the difference is estimated in terms of 
the Navier-Stokes solution, on which Kato's criterion was imposed. In contrast, our proof
estimates the difference in terms of the full-space Euler solution, which is smooth in the
context of interest.

\section{Compactly supported initial vorticity}

Now that we are in possession of Theorem \ref{mainthm} we will examine two 
asymptotic problems for which we can prove the convergence condition on the initial 
velocity.  We focus on flows with compactly supported vorticity,  and the diameter
of the support of vorticity becomes the order one length scale, relative to which the 
obstacle is small. 

Let us begin with the three dimensional case. We consider an initial vorticity 
$\omega_0$ which is assumed to be smooth, compactly supported in $\real^3\setminus\{0\}$, 
and divergence-free. Let $\vare > 0$ be sufficiently small so that the support of $\omega_0$ 
is contained in $\Pi_{\vare}$. The domain $\Pi_{\vare}$ is assumed simply connected so 
that there exists a unique divergence-free vector field, tangent to $\partial \Pi_{\vare}$, 
in $L^2(\Pi_{\vare})$, whose curl is $\omega_0$, see, for example, \cite{iftimie}.  
We take $\vartheta^{\vare}$ to be this unique vector field. We take $u_0$ to be the 
unique divergence-free vector field in $L^2(\real^3)$ whose curl is $\omega_0$, 
given by the full space Biot-Savart law. 

In \cite{iftimie}, D. Iftimie and J. Kelliher studied the small obstacle asymptotics for viscous
flow in $\Pi_{\vare}$, for fixed viscosity, in three dimensions. They proved that the 
small obstacle limit converges to the appropriate Leray solution of the Navier-Stokes 
equations in the full space. One important ingredient in their proof was precisely to
verify strong convergence of the initial data; in our notation, Iftimie and Kelliher proved that
\[\|\vartheta^{\vare} - u_0\|_{L^2(\real^3)} =O(\vare^{\frac32}),\]
as $\vare \to 0$. 

We may hence apply Theorem \ref{mainthm} to obtain the following corollary. 

\begin{corollary} Let $\omega_0 \in C^{\infty}_c(\real^3 \setminus \{0\})$ and consider
$u_0$ and $\vartheta^{\vare}$ defined as above. Fix $T>0$ and assume that the solution 
$u = u(x,t)$ of the incompressible Euler equations in $\real^3$, with initial velocity 
$u_0$, exists up to time $T$. Let $u^{\neps}$ be a Leray solution of \eqref{neps} with 
initial velocity $\vartheta^{\vare}$. Then there exist constants 
$C_1 = C_1(\Omega,\omega_0,T) > 0$ and $C_2=C_2(\Omega,\omega_0,T) > 0$ such that 
\[\|u^{\neps}(\cdot,t) - u(\cdot,t)\|_{L^2(\real^3)} \leq C_2 \sqrt{\nu},\]
for all $0<\vare<C_1 \nu$ and all $0\leq t \leq T$.
\end{corollary}
 
Next we discuss at length the case $n=2$. In dimension two the exterior domain is no
longer simply connected. This means that the vorticity formulation of the Euler equations
is incomplete, and we must specify the harmonic part of the initial velocity as well as
the initial vorticity, see \cite{iln03} for a thorough discussion of this issue.  
To specify the asymptotic problem we wish to consider, we must choose the initial data
for \eqref{neps}.

Let $\omega_0$ be smooth and compactly supported in $\real^2 \setminus \{0\}$. Let $K$ denote
the operator associated with the Biot-Savart law in the full plane and set $H = x^{\perp}/(2\pi|x|^2)$,
to be its kernel. Let $K^{\vare}$ be the operator associated with the Biot-Savart in $\Pi_{\vare}$, i.e.,
$K^{\vare} = \nabla^{\perp} \Delta^{-1}_{\Pi_{\vare}}$, where $\Delta_{\Pi_{\vare}}$ is the
Dirichlet Laplacian in $\Pi_{\vare}$. Let $H^{\vare}$ be the generator of the
harmonic vector fields in $\Pi_{\vare}$, normalized so that its circulation around
$\partial \Pi_{\vare}$ is one. The divergence-free vector fields in $\Pi^{\vare}$ with curl equal to $\omega_0$ are
of the form $K^{\vare}[\omega_0] + \alpha H^{\vare}$, with $\alpha \in \real$, see \cite{iln03}. 
In \cite{iln06} the authors studied the asymptotic behavior, as $\vare \to 0$, of solutions of \eqref{neps} with
$\nu$ fixed and initial velocity $K^{\vare}[\omega_0] + \alpha H^{\vare}$. It was shown in \cite{iln06}
that $u^{\neps}$ converges to a solution of the Navier-Stokes equations in the full plane with
initial data $K[\omega_0] + (\alpha - m)H$, where $m = \int \omega_0$, as long as $|\alpha - m|$ is sufficiently
small. 

For the vanishing viscosity limit, we must consider only the case $\alpha = m$. There are two reasons
for this. First, $K^{\vare}[\omega_0] + \alpha H^{\vare}$ converges weakly to 
$K[\omega_0] + (\alpha - m)H$ in distributions, see Lemma 10 in \cite{iln06}, but, as we shall see,
this convergence is not strong in $L^2$ (see Remark 1 following the proof of the next lemma). 
Second, one cannot expect solutions of the Euler equations in the full plane with initial velocity 
$K[\omega_0] + (\alpha - m)H$ to be smooth (even existence is not clear) unless $\alpha = m$.   

In view of this discussion, set $u_0 = K[\omega_0]$ and $\vartheta^{\vare} = K^{\vare}[\omega_0] + m H^{\vare}$.
With this notation, we can prove strong convergence in $L^2$ of the initial data, as follows. 

\begin{lemma} \label{indata}
Fix $\vare_0$ such that the support of $\omega_0$ does not intersect $\Omega_{\vare}$ for 
any $\vare < \vare_0$. There exists a constant $C>0$, depending on
$\Omega$ and $\omega_0$ 
such that 
\[\|\vartheta^{\vare} - u_0\|_{L^2(\real^2)} \leq C \vare.\]
\end{lemma}

\begin{proof}
We begin the proof with a construction whose details can be found in \cite{iln03}. In Section 2 of \cite{iln03},
an explicit formula for both $K^{\vare}$ and $H^{\vare}$ can be found in terms of a conformal
map $T$, which takes $\Pi$ into the exterior of the unit disk centered at zero. The construction of
$T$ and its behavior near infinity are contained in Lemma 2.1 of \cite{iln03}.  Using identities
(3.5) and (3.6) in \cite{iln03}, we have that the vector
field $H^{\vare}$ can be written explictly as
\[H^{\vare} = H^{\vare}(x) = \frac{1}{2\pi\vare} DT^t(x/\vare)\, \frac{(T(x/\vare))^{\perp}}{|T(x/\vare)|^2}.\]
and the operator $K^{\vare}$ can be written as an integral operator with kernel $\cal{K}^{\vare}$, given by
\[\cal{K}^{\vare} = \frac{1}{2\pi\vare} DT^t(x/\vare) 
\Bigl(\frac{(T(x/\vare) - T(y/\vare))^{\perp}}{|T(x/\vare) - T(y/\vare)|^2}
- \frac{(T(x/\vare) - (T(y/\vare))^{\ast})^{\perp}}{|T(x/\vare) - (T(y/\vare))^{\ast}|^2}\Bigr),\]
where $x^{\ast} = x/|x|^2$ denotes the inversion with respect to the unit circle.
Furthermore, we recall Theorem 4.1 of \cite{iln03}, from which we obtain 
\begin{equation*} 
\|\vartheta^{\vare}\|_{L^{\infty}(\Pi_{\vare})} \leq C\|\omega_0\|_{L^{\infty}}^{1/2} \|\omega_0\|_{L^1}^{1/2},
\end{equation*}
for some constant $C>0$.

To understand the behavior for $\vare$ small in the expressions above, we need to understand the behavior
of $T(x)$ for large $x$. We use Lemma 1 in \cite{iln06b}, which is a more detailed version
of Lemma 2.1 in \cite{iln03}, to find that there exists a constant $\beta>0$ such that
\begin{equation}
  \label{**}
T(x/\vare) = \beta x \vare^{-1} + h(x/\vare),  
\end{equation}
with $h=h(x)$ a bounded, holomorphic function on $\Pi_1$ satisfying $|Dh(x)| \leq C/|x|^2$. 
Therefore,
\begin{equation}
  \label{*}
| DT(x/\vare) - \beta \mathbb{I} | \leq  C \frac{\vare^2}{|x|^2}.  
\end{equation}

We will need a further estimate on the bounded holomorphic function $h = h(z) = T(z)-\beta z$, namely that
\begin{equation}
  \label{relh}
|h(z_1) - h(z_2)| \leq C\frac{|z_1 - z_2|}{|z_1||z_2|},  
\end{equation}
for some constant $C>0$ independent of $z_1$, $z_2$. This estimate holds since, by construction (see Lemma 2.1 in 
\cite{iln03}), we have that $h(z) = g(1/z)$ with $g$ a holomorphic
function on $(\Pi_1)^*$, whose derivatives are bounded in the closure
of $(\Pi_1)^*$. Here, $(\Pi_1)^*$ denotes the image of $\Pi_1$ through
the mapping $x\mapsto x^{\ast} = x/|x|^2$ to which we add
$\{0\}$. Here, we are using the following fact: If $D$ is a bounded domain with $\partial D$ a $C^1$ Jordan curve
then any bounded function $f \in C^1(D)$ with bounded derivatives is globally Lipschitz in $D$. This fact
is a nice exercise in basic analysis, which we leave to the reader.


Therefore we have
\begin{equation}
  \label{***}
|h(z_1) - h(z_2)| = \Bigl|g\Bigl(\frac{1}{z_1}\Bigr) - g\Bigl(\frac{1}{z_2}\Bigr)\Bigr| \leq 
C\Bigl|\frac{1}{z_1} - \frac{1}{z_2}\Bigr| = C\frac{|z_1 - z_2|}{|z_1||z_2|}.  
\end{equation}

In order to estimate $\|\vartheta^{\vare} - u_0\|_{L^2(\Pi_{\vare})}$ we use the fact that the support of $\omega_0$ is contained in $\Pi_{\vare}$ for $\vare$ sufficiently small $\vare$ to write
\begin{multline*}
2\pi[ \vartheta^{\vare}(x) - u_0(x)] = 
\int_{\Pi_{\vare}} \Bigl( \frac{1}{\vare}DT^t(x/\vare) \, \frac{(T(x/\vare) - T(y/\vare))^{\perp}}{|T(x/\vare) - T(y/\vare)|^2}- \frac{(x-y)^{\perp}}{|x-y|^2}\Bigr)\omega_0(y)\,dy \\
+ \int_{\Pi_{\vare}} \frac{1}{\vare} DT^t(x/\vare) \Bigl(\frac{(T(x/\vare))^{\perp}}{|T(x/\vare)|^2}- \frac{(T(x/\vare) - (T(y/\vare))^{\ast})^{\perp}}{|T(x/\vare) - (T(y/\vare))^{\ast}|^2}\Bigr)\omega_0(y)\,dy
\equiv \mathfrak{A^{\vare}} + \mathfrak{B^{\vare}}.
\end{multline*}

Let us begin by estimating $\bbug$. We make the change of variables $\eta = \vare T(y/\vare)$, whose Jacobian is 
$J = |\det (DT^{-1})(\eta/\vare)|$, a bounded function. Additionally, we set $z = \vare T(x/\vare)$. With this we find:
\[\bbug = DT^t(x/\vare)  \int_{\{|\eta| > \vare \}} 
\Bigl(\frac{z^{\perp}}{|z|^2}- \frac{(z - \vare^2 \eta^{\ast})^{\perp}}{|z - \vare^2\eta^{\ast}|^2}\Bigr)
\omega_0(\vare T^{-1}(\eta/\vare))\,J d\eta.\]
We observe now that there exists $\rho$ independently of $\vare$ such
that the support of $\omega_0(\vare T^{-1}(\eta/\vare))$ is contained in the set $\{|\eta| >
\rho\}$. 
Moreover, one can bound $|z-\vare^2\eta^\ast|\geq |z|-\vare^2|\eta^\ast|\geq |z|-\vare^2/\rho\geq |z|/2$ provided that $\vare^2\leq\rho/2$. Therefore we can write
\[|\bbug| \leq C \int_{\{|\eta| > \rho \}} \frac{\vare^2 |\eta^{\ast}|}{|z||z-\vare^2\eta^{\ast}|} 
|\omega_0(\vare T^{-1}(\eta/\vare))|\,J d\eta\leq C \frac{\vare^2}{|z|^2},\]
where $C$ depends on the support of $\omega_0$, on the $L^1$-norm of $\omega_0$ and on the domain $\Omega$ through the bounds on the conformal map $T$ and its derivatives. Finally, we use this estimate in the integral of the square of $\bbug$:
\[\int_{\Pi_{\vare}}|\bbug|^2\,dx \leq C\vare^4 \int_{\{ |z| > \vare \}} \frac{1}{|z|^4}\,dz \leq C \vare^2,\]
as desired. 

Next we treat $\abug$. First we re-write $\abug$ in a more convenient form:
\[\abug = \int_{\Pi_{\vare}} \frac{1}{\beta} DT^t(x/\vare) \Bigl(
\frac{\beta}{\vare} \frac{(T(x/\vare) - T(y/\vare))^{\perp}}{|T(x/\vare) - T(y/\vare)|^2}
- \frac{(x-y)^{\perp}}{|x-y|^2}\Bigr)\omega_0(y)\,dy \]

\[+ \int_{\Pi_{\vare}} \Bigl(\frac{1}{\beta} DT^t(x/\vare) - \mathbb{I} \Bigr)
 \frac{(x-y)^{\perp}}{|x-y|^2} \,\omega_0(y)\,dy \]

\[\equiv \abug_1 + \abug_2.\]

By \eqref{*}, the term $\abug_2$ can be easily estimated:
\[|\abug_2| \leq \frac{\vare^2}{|x|^2} \int_{\Pi_{\vare}} \frac{1}{|x-y|} \,|\omega_0(y)|\,dy \leq C
\frac{\vare^2}{|x|^2},\]
so this reduces to an estimate similar to the one we found for $\bbug$. 

Next we examine $\abug_1$. We use the expression for $T$ given in
\eqref{**} to write
\[\abug_1 = \frac{DT^t(\frac{x}{\vare})}{\beta} \int_{\Pi_{\vare}}  \Bigl(
\frac{(x-y + (\frac{\vare}{\beta})[h(\frac{x}{\vare})-h(\frac{y}{\vare} )])^{\perp}}{
|x-y + (\frac{\vare}{\beta})[h(\frac{x}{\vare})-h(\frac{y}{\vare})]|^2} - \frac{(x-y)^{\perp}}{|x-y|^2}\Bigr)\omega_0(y)\,dy. \]
With this we have:
\[|\abug_1| \leq C \int_{\Pi_{\vare}}\frac{|\frac{\vare}{\beta}[h(\frac{x}{\vare})-h(\frac{y}{\vare} )]|}
{|x-y + (\frac{\vare}{\beta})[h(\frac{x}{\vare})-h(\frac{y}{\vare})]||x-y|}\,|\omega_0(y)|\,dy.\]
We will make use several times of the estimate we obtained for $h$
given in \eqref{***}. First
\begin{equation} \label{Deltah}
|\frac{\vare}{\beta}|[h(\frac{x}{\vare})-h(\frac{y}{\vare} )]| \leq C\frac{\vare^2|x-y|}{|\beta||x||y|}.
\end{equation}
Using \eqref{Deltah} gives
\[ |\abug_1| \leq C \int_{\Pi_{\vare}}\frac{\vare^2}
{|x-y + (\frac{\vare}{\beta})[h(\frac{x}{\vare})-h(\frac{y}{\vare})]||\beta||x||y|}\,|\omega_0(y)|\,dy.\]
Let $R$, $r>0$ be such that the support of $\omega_0$ is contained in the disk of radius $R$ and outside the disk  of radius $r$. We will estimate $\abug_1$ in two regions: $|x| \geq 2R$ and $|x| < 2R$. Also, recall that the estimate of the 
$L^2$-norm of $\abug_1$ is to be performed in $\Pi_{\vare}$ so we may assume throughout that $|x| \geq C\vare$. Suppose first that $|x|\geq 2R$. Then we find:
\[|\abug_1| \leq C\frac{\vare^2}{|x|^2}.\]
Above we used that $r<|y|\leq|x|/2$ and hence $|x-y + (\frac{\vare}{\beta})[h(\frac{x}{\vare})-h(\frac{y}{\vare})]| \geq C|x|$ if $\vare$ is sufficiently small, since $h$ is bounded. Finally, in the region $C\vare \leq |x| < 2R$ we use \eqref{relh} and the fact that $|y|$ is of order 1  to bound
\begin{equation*}
|x-y + (\frac{\vare}{\beta})[h(\frac{x}{\vare})-h(\frac{y}{\vare})]
\geq  ||x-y|-\vare^2(|x-y|/|x||y|)|\geq \frac{|x-y|}2 
\end{equation*}
for $\vare$ small enough.
Therefore
\[|\abug_1|\leq C\frac{\vare^2}{|x|} \int_{\Pi_{\vare}}\frac{|\omega_0(y)|}{|x-y|} \,dy \leq C\frac{\vare^2}{|x|}\leq C\vare. 
 \]
Clearly this last portion has $L^2$-norm in the region $|x|<2R$ bounded by $C\vare$.

\end{proof}

{\bf Remark 1:} Let $\omega_0 \in C^{\infty}_c(\real^2 \setminus \{0\})$ and $\alpha \in \real$, $\alpha \neq m$. 
We observe that $K^{\vare}[\omega_0] + \alpha H^{\vare}$ does not converge strongly in $L^2$ to $K[\omega_0]
+ (\alpha-m)H$. We argue by contradiction, assuming this convergence holds. In view of the Lemma above, 
$K^{\vare}[\omega_0] + m H^{\vare}$ converges strongly to $K[\omega_0]$ so we must have
\[(\alpha - m)H^{\vare} \to (\alpha - m)H.\]
This does not hold, as it can be easily seen in the case of the exterior of the disk. In this case,
$H^{\vare} = H$ outside the disk of radius $\vare$, but $H^{\vare}$ vanishes for $|x| < \vare$.
Since $\|H\|_{L^2(\{|x|<\vare\})} = \infty$, we have a contradiction. 

{\bf Remark 2:} Note that if we were willing to confine our analysis to the exterior of a small disk,
the proof of Lemma \ref{indata} would be much simpler. 
Indeed, let $\Omega_{\vare} = \{|x|>\vare\}$. Then the conformal map $T$
is the identity, so $\abug \equiv 0$ and all that is needed is the easier estimate for $\bbug$. 

{\bf Remark 3:} The constant $\alpha - m$ is precisely the circulation of $\vartheta^{\vare}$ 
around the boundary of $\Pi_{\vare}$. The condition $\alpha - m = 0$ is physically reasonable, in
particular because viscous flows vanish at the boundary, and therefore, so does their circulation.
This is the condition for the small obstacle limit of ideal flow to satisfy Euler equations in the
full plane, see \cite{iln03} and also for the small obstacle limit of viscous flows to satisfy
the full plane Navier-Stokes equations for all viscosities, see \cite{iln06}. The argument in
\cite{iln06} required sufficiently small $\alpha - m$ to obtain the appropriate limit when 
$\vare \to 0$, and the smallness condition was actually $\alpha - m = \mathcal{O}(\nu)$ as 
$\nu \to 0$. 

We conclude this section with the formal statement of a corollary which encompasses Theorem \ref{mainthm}
and Lemma \ref{indata}.

\begin{corollary}
Let $\omega_0 \in C^{\infty}_c(\real^2 \setminus \{0\})$ and consider
$u_0$ and $\vartheta^{\vare}$ defined as in Lemma \ref{indata}. Let  
$u = u(x,t)$ be the global smooth solution of the incompressible Euler equations 
in $\real^2$, with initial velocity  $u_0$. Let $u^{\neps}$ be the solution of \eqref{neps} with 
initial velocity $\vartheta^{\vare}$. Fix $T>0$. There exist constants 
$C_1 = C_1(\Omega,\omega_0,T) > 0$ and $C_2=C_2(\Omega,\omega_0,T) > 0$ such that 
\[\|u^{\neps}(\cdot,t) - u(\cdot,t)\|_{L^2(\real^2)} \leq C_2 \sqrt{\nu},\]
for all $0<\vare<C_1 \nu$ and all $0\leq t \leq T$.

\end{corollary}

\section{Physical interpretation and conclusions}

      The behavior of incompressible viscous flow past a bluff body, such as a long cylinder
or a sphere is a classical problem in fluid dynamics, to the extent of having conference series
devoted to it, see {\it http://www.mae.cornell.edu/bbviv5/}. Let us consider the simplest situation, 
two-dimensional flow of a viscous fluid with kinematic viscosity $\mu$, filling the whole plane minus 
a disk of diameter $L$, with constant driving velocity $U$ at infinity.  The disturbance caused by the 
disk, known as its {\it wake}, depends only on the {\it Reynolds number} associated with the flow, 
given by 
\[ Re  \equiv \frac{LU}{\mu}. \]
The observed behavior of the wake begins, for small $Re$, as a steady solution of the Navier-Stokes equations,
but the wake undergoes a series of bifurcations as $Re$ grows, progressively developing steady recirculation zones ($4 < Re < 40$, periodic recirculation and a Von Karman street ($40 < Re < 200$), nonperiodic vortex shedding 
($200 < Re < 400$), leading to turbulence ($Re > 400$). See \cite{panton}, Section 15.6, for details and illustrations.  

	In our problem, which involves nearly inviscid flow past a {\it small} bluff body, the qualitative behavior of
the wake of the small obstacle is determined by the {\it local} Reynolds number, which encodes the way in which
an observer at the scale of the obstacle experiences the flow. Basically, by making our obstacle small, we are making
the flow more viscous at its scale. We assume that the Navier-Stokes system under consideration, \eqref{neps}, is dimensional, i.e. has time and space measured in seconds and meters, and mass 
normalized so that fluid density is one. In these units, the kinematic viscosity for air is $\mu = 14.5 \times 10^{-6} m^2/sec$, and for water it is $\mu = 1.138 \times 10^{-6} m^2/sec$, both at $15^o C$. 

Let us restrict our discussion to the two-dimensional case. The smallness condition in Theorem \ref{mainthm}, \eqref{smcond}, reads
$\vare < C_1 \mu,$ and the {\it dimensional} constant $C_1$, requires closer scrutiny. Actually, the constant $C_1$
is given by:
\[C_1 = \frac{1}{8 K_4 K_6^2}, \]
where $K_4$ appears in Lemma \ref{epsests}, item (4), and $K_6$ is from Lemma \ref{poincare}.
$K_6$ is a non-dimensional constant that depends on the shape of the obstacle 
$\Omega$. The constant $K_4$ can be chosen as
\[K_4 = \sup_{x\in \real^n , \, t \in [0,T]} \vare(|\nabla \psi \nabla \varphi^{\vare}|(x,t)
+ |\psi \nabla^2 \varphi^{\vare}|(x,t)).\]

The function $\psi$ above is the stream function of the full-plane Euler flow, adjusted so
that $\psi(0,t) = 0$. Also, $\vare \nabla \varphi^{\vare}$ is $\mathcal{O}(1)$, localized 
near the obstacle and $\vare \nabla^2 \varphi^{\vare}$ is $\mathcal{O}(1/\vare)$, also localized 
near the obstacle. Therefore, both terms included in $K_4$ are associated with first derivatives
of the stream function at the obstacle, i.e. with the local velocity $u(0,t)$.  Therefore,
we can write $K_4 = \widetilde{K}_4 \sup_{t}|u(0,t)|$ (we can also assume that the limiting
Euler flow is stationary, to avoid the time dependence).

From the point of view of the obstacle, the inviscid velocity $u(0,t)$ acts as a constant (in space) forcing velocity imposed at infinity, and therefore, the qualitative behavior of the wake of the obstacle is determined 
by the local Reynolds number $Re_{\loc} \equiv u(0,t) \vare / \mu$. Clearly, condition \eqref{smcond} can be rewritten
as 
\[ Re_{\loc} < \frac{1}{8 \widetilde{K}_4 K_6^2}.\]
The non-dimensional constant $K_6$ is related to the constant in the Poincar\'{e} inequality in the unit disk.
Examining our proof for the case of the disk, we cannot make the constant $\widetilde{K}_4 K_6^2$ smaller 
than something of the order of $10$. Therefore, our result is restricted to rather viscous wakes.          

When it occurs, the turbulence is caused by vorticity shed by the obstacle through boundary layer 
separation. The main difficulty  in studying the vanishing viscosity limit in the presence of 
boundaries is the fact that, although the Navier-Stokes equations do have a vorticity form, valid 
in the bulk of the fluid, the vorticity equation does not satisfy a useful boundary condition, so 
that we cannot control the amount of vorticity added to the flow by the boundary layer. In the 
proof of Theorem \ref{mainthm}, we found a way of controlling the kinetic energy of the wake without 
making explicit reference to the vorticity. At this point, it is reasonable to ask whether we can 
control the vorticity content of the wake as well. To answer that, we introduce the {\it enstrophy} 
$\Omega^{\neps}(t)$ of the flow:
\[\Omega^{\neps}(t) \equiv \frac{1}{2} \int_{\Pi_{\vare}} |\mbox{ curl } u^{\neps}|^2 \, dx.\] 
Of course, enstrophy measures how much vorticity is in the flow, but its behavior as $\nu \to 0$ is 
also involved in the statistical structure of a turbulent wake.

\begin{corollary} 
For any $T>0$ there exists a constant $C>0$, independent of $\nu$ such that
\[\int_0^T \Omega^{\neps}(t) \, dt \leq C.\]
\end{corollary}

\begin{proof}
We go back to relations \eqref{bliblu} and \eqref{dty} and include the viscosity term which had been ignored. 
We find:
\begin{equation*}
\frac{dy}{dt} +\frac\nu4\|\nabla W^{\neps}\|_{L^2}^2 \leq C'_1\nu + C'_2 y.  
\end{equation*}

We next integrate in time to obtain
\[ \|W^{\neps}(\cdot,T)\|_{L^2}^2 - \|W^{\neps}(\cdot,0)\|_{L^2}^2 +
\frac{\nu}{4} \int_0^T \|\nabla W^{\neps}(\cdot,t)\|_{L^2}^2 \, dt\]
\[\leq  C'_1 T \nu + C'_2 \int_0^T \|W^{\neps}(\cdot,t)\|_{L^2}^2 \, dt.\]
Now we use Theorem \ref{mainthm} and ignore a term with good sign to obtain
\[\frac{\nu}{4} \int_0^T \|\nabla W^{\neps}(\cdot,t)\|_{L^2}^2 \, dt \leq C T \nu + 
\|W^{\neps}(\cdot,0)\|_{L^2}^2  \leq C' T \nu,\]
where we used Lemma \ref{indata} together with item (3) from Lemma \ref{epsests} to
estimate the initial data term. From this we conclude that 
\[ \int_0^T \|\nabla W^{\neps}(\cdot,t)\|_{L^2}^2 \, dt \leq C.\]
Finally, we observe that 
\[\Omega^{\neps} \leq C\|\nabla W^{\neps}\|_{L^2}^2 +
   C\|\nabla u^{\vare}\|_{L^2}^2 \leq C \|\nabla W^{\neps}\|_{L^2}^2+ CK_1,\]
by item (1) in Lemma \ref{epsests}. This concludes the proof.

\end{proof} 

Finally, let us consider some open questions naturally associated with the research
presented here. First, one would like to weaken, and ultimately remove, the smallness condition 
on the size of the obstacle; this is the most physically interesting follow-up problem.
Second, one would also like to consider two dimensional flows with nonzero initial circulation 
at the obstacle, in order to study the interaction of the vanishing viscosity and vanishing obstacle limits
in more detail. This would improve the connection of the present work with the authors' previous 
results in \cite{iln03,iln06}. An easier version of this second problem would be to consider 
an initial circulation of the form $\gamma = \gamma(\nu)$ and find out how fast $\gamma$ has to 
vanish as $\nu \to 0$ in order to retain our result. A third problem is to describe more precisely 
the asymptotic structure of the difference between the full-space Euler flow and the approximating
small viscosity, small obstacle flows.         

{\footnotesize  {\it Acknowledgments:} 
This work was done  during the Special Semester in Fluid Mechanics at the Centre Interfacultaire Bernoulli, EPFL; the authors wish to express their gratitude for the hospitality received. The authors would like to thank Peter Constantin,
Jim Kelliher and Franck Sueur for many helpful comments.}

\end{document}